\documentclass[12pt]{amsart}
\usepackage{amssymb,latexsym, amsmath, amscd, array, graphicx}

\swapnumbers \numberwithin{equation}{section}

\theoremstyle{plain}

\newtheorem{thm}{Theorem}[section]

\newtheorem{prop}[thm]{Proposition}
\newtheorem{cor}[thm]{Corollary}

\theoremstyle{definition}

\newtheorem{remark}[thm]{Remark}

\DeclareMathOperator{\as}{{\rm asdim}}

\def\Q{{\mathbb Q}}

\def\1{\hbox{\rm\rlap {1}\hskip.03in{\rom I}}}
\def\Bbbone{{\rm1\mathchoice{\kern-0.25em}{\kern-0.25em}
{\kern-0.2em}{\kern-0.2em}I}}

\long\def\forget#1\forgotten{} %

\newcommand\ver[1]{\marginpar{\tiny Changed in Ver \VER}}

\date{\today}
\begin{document}

\title[Embedding]{Embedding of  Higson compactification into the product of adelic solenoids.}

%
\author[A.~Dranishnikov]{Alexander  Dranishnikov}
\author[J.~Keesling]{James  Keesling} 

\thanks{The first author was supported by Simons Foundation}

\address{Alexander N. Dranishnikov, Department of Mathematics, University
of Florida, 358 Little Hall, Gainesville, FL 32611-8105, USA}
\email{dranish@ufl.edu}
\address{James E. Keesling, Department of Mathematics, University
of Florida, 358 Little Hall, Gainesville, FL 32611-8105, USA}
\email{kees@ufl.edu}


\maketitle

\begin{abstract}
The Higson compactification of any simply connected proper geodesic  metric space admits an embedding into a product of adelic solenoids that induces an isomorphism of 1-dimensional  cohomology.
\end{abstract}

\section{Introduction}

The Higson corona of a metric space was introduced by Higson and Roe~\cite{Ro1},\cite{Ro2} to bring an approximation of the $K$-theory of its Roe algebra $C_{roe}^*(X)$ by the topological $K$-theory of the corona.
The $K$-theory of the Roe algebra is the recipient in the coarse assembly map and therefore, it plays an important role in a coarse approach to the Novikov and the Baum-Connes conjectures.
It turns out that for injectivity of  the coarse assembly map it is sufficient to have the Higson compactification acyclic~\cite{Ro1}. James Keesling noticed that the Higson compactification
of $\mathbb R^n$ is not acyclic in dimension one~\cite{K}. Later Dranishnikov and Ferry have constructed nontrivial second cohomology classes of the Higson compactification of $\mathbb R^2$~\cite{DF}.

It turns out that all the above cohomology classes of the Higson compactifications are divisible. This leaves a hope that all cohomology groups of the Higson compactification  are vector spaces over $\mathbb Q$. In that case the Higson compactification would be acyclic with respect to any finite coefficient group.
In~\cite{DFW} Dranishnikov, Ferry, and Weinberger proved the Novikov Higher Signature conjecture for manifolds whose fundamental group $\Gamma$ admits a finite classifying space $B\Gamma$ with trivial \v{C}ech cohomology of the Higson compactification $\overline{E\Gamma}$ of the universal cover $E\Gamma$ with $\mathbb Z_2$ coefficients.
In the same paper it was proven that $\check H^*(\overline{E\Gamma};\mathbb Z_p)=0$ for all prime $p$ whenever the asymptotic dimension $\as\Gamma$ of the group $\Gamma$  is finite.

In this paper we construct an embedding of $\overline{E\Gamma}$ into the product of adelic solenoids $\prod\Sigma_p$ which induces isomorphism of rational cohomology. In particular, our embedding implies that $\check H^1(\overline{E\Gamma})=\oplus\mathbb Q$ and, hence, $\check H^1(\overline{E\Gamma};\mathbb Z_p)=0$ for all $p$.

We think that it is a reasonable  to believe that our embedding induces an epimorphism of
cohomology groups for higher dimensions. This conjecture seems plausible  in the case when $\dim\overline{E\Gamma}<\infty$. We note that from~\cite{DFW} one can derive our conjecture for groups with $\as\Gamma<\infty$. We recall that always $\dim\nu(E\Gamma)\le\as\Gamma$~\cite{DKU} and  $\dim\nu(E\Gamma)=\as\Gamma$~\cite{Dr} when the latter is finite. Here $\nu(E\Gamma)=\nu\Gamma$ is the Higson corona: $\nu(E\Gamma)=\overline{E\Gamma}\setminus E\Gamma$.
There is a long open question whether  there exists a  metric spaces $X$ with $\dim\nu X<\infty$  and $\as\Gamma=\infty$. Perhaps a more challenging related question here is whether
the Novikov conjecture holds true for 
groups with $\dim\nu\Gamma<\infty$. We recall that the Novikov conjecture for groups  with $\as\Gamma<\infty$  was proven more than two decades ago in~\cite{Yu}.

\section{Preliminaries}

\subsection{Higson compactification} Let $(X,d)$ be a metric space with a base point $x_0$. By $B(r,x)$ we denote the open ball of radius $r$ with the center in $x$.
We call a map  $f:X\to Y$ to a metric space $Y$ {\em slowly oscillating} if for any $r>0$ the diameter of $f(B(x,r))$ tends to 0 as $d(x,x_0)$ tends to infinity.
The Higson compactification $\bar X$ of $X$ is characterized by bounded slowly oscillating continuous functions $f:X\to\mathbb R$. If $C_h$ is the ring of all such functions, then $\bar X$ is homeomorphic to the closure of $X$ under the embedding 
$$(f)_{f\in C_h}:X\to\prod_{f\in C_h}[\inf f,\sup f].$$
We use notation $\nu X=\bar X\setminus X$ for the Higson corona. The following is straightforward.
\begin{prop}\label{Higson}
For a metric space $X$ a continuous map $f:X\to C$ to a compact metric space $C$ extends continuously to the Higson corona if and only if
$f$ is slowly oscillating.
\end{prop}

A metric space $X$ is called {\em proper} if the distance to each point in $X$ is a proper function. A metric space $X$ is called {\em geodesic} if for any pair of points $x,y\in X$ there is an isometric embedding $\xi:[0,d(x,y)]\to X$ with $\xi(0)=x$ and $\xi(d(x,y))=y$.

The following can be found in~\cite{DKU}.
\begin{prop}\label{Higson-2}
In a proper metric space $X$ closed subsets $Y$ and $Y'$ have common points in the Higson corona, $Cl_{\bar X}Y)\cap Cl_{\bar X}(Y')\cap\nu X\ne \emptyset$,  if and only if
$d(Y\setminus B(x_0,r),Y')\to\infty$ as $r\to\infty$.
\end{prop}

\

\subsection{Solenoids} We start from the definition of the $p$-adic solenoid $\Sigma_p$ as the inverse limit:
 $$\Sigma_p=\lim_{\leftarrow}\{S^1 \stackrel{p}\leftarrow S^1\stackrel{p}\leftarrow S^1\leftarrow\cdots\}$$ 
of unit circles $S^1\subset\mathbb C$ where the bonding maps are group homomorphisms defined by  the map $z\mapsto z^p$. Thus, $\Sigma_p$ is an abelian topological group. We use $+$ for the group operation in $\Sigma_p$.
 Let $p_0:\Sigma_p\to S^1$ be the projection onto the first circle in the inverse sequence. Note that the subgroup $p_0^{-1}(1)=A_p\subset\Sigma_p$ is the group of $p$-adic integers.
We use the notation $A_p$ for $p$-adic integers instead of more common notation from algebra $\mathbb Z_p$, since the later is used in topology (and already was used in this paper) for the $\mathbb Z$ $mod\ p$ group.

Let $u:\mathbb R\to S^1\subset\mathbb C$ be the universal cover given by the exponential map $u(t)=e^{2\pi it}$. Thus, $u$ is a group homomorphism
with $u^{-1}(1)=\mathbb Z$.
Then the  lift $\tilde u:\mathbb R\to\Sigma_p$ that takes $0$ to the unit of $\Sigma_p$ is an injective group homomorphism. Thus, the group of reals is a subgroup
of the solenoid, $\mathbb R\subset\Sigma_p$. We note that $\mathbb R\cap A_p=\mathbb Z$. Since $\mathbb Z$ is dense in $A_p$, it follows that $\mathbb R$ is dense in $\Sigma_p$. Also
we note that all path components of $\Sigma_p$ are the $\mathbb R$-cosets $a+\mathbb R$, $a\in\Sigma_p$. They are the bijective images of the reals $g:\mathbb R\to\Sigma_p$ that are obtained as different  lifts of the universal cover $u$ with respect to $p_0$. Another way to define the $p$-adic solenoid is taking the qoutient group $\Sigma_p=(\mathbb R\times A_p)/\mathbb Z$ for the natural diagonal embedding $\mathbb Z\to\mathbb R\times\mathbb A_p$.

The adelic solenoid is the quotient $\Sigma_\mathbb Q=\mathbb A_{\mathbb Q}/\mathbb Q$ of the adels of the rationals~\cite{BV}. It can be viewed as the quotient group
$\Sigma_\mathbb Q=(\mathbb R\times \hat{\mathbb Z})/\mathbb Z$ where $\hat{\mathbb Z}$ is the profinite completion of the integers and $\mathbb Z\to\mathbb R\times\hat{\mathbb Z}$
is the diagonal map for natural inclusions $\mathbb Z\subset\mathbb R$ and $\mathbb Z\subset\hat{\mathbb Z}$. We note that the image of the inclusion $\mathbb Z\subset\hat{\mathbb Z}$ is dense as well as the image of the inclusion $\mathbb R\subset\Sigma_\mathbb Q$. The adelic solenoid admits the following description as the inerse limit
$$\Sigma_\mathbb Q=\lim_{\leftarrow}\{S^1 \stackrel{z^2}\leftarrow S^1\stackrel{z^3}\leftarrow S^1\stackrel{z^4}\leftarrow\cdots\}$$
where the bonding homomorphism from the $n$th circle to the previous is taking to the power $n$, $z\mapsto z^n$. Again we denote by $p_0:\Sigma_\mathbb Q\to S^1$ the projection to the first circle. All the above facts about $\Sigma_p$ hold true for $\Sigma_\mathbb Q$ with replacement of $p$-adic integers by the profinite completion of $\mathbb Z$.
In this paper we will be using the notations $\Sigma_\infty$ for $\Sigma_\mathbb Q$ and $A_\infty$ for $\hat{\mathbb Z}$.

\

\begin{prop}\label{kees}
For the Higson compactification $\bar X$ of a simply connected geodesic metric space $X$ there is an isomorphism $$\check H^1(\bar X)=\bigoplus_{\mathcal A}\mathbb Q.$$
\end{prop}
\begin{proof} The idea of the  proof is taken from~\cite{K}.
It suffices to show that each element $\alpha\in\check H^1(\bar X)$ is $p$-divisible for all $p$. Let $\phi:\bar X\to S^1$ be a map representing $\alpha$, $\alpha=\phi^*(1)$, where
$1\in\mathbb Z=H^1(S^1)$. Since $X$ is simply connected, there is a lift $f:X\to\Sigma_p$ with respect to $p_0$.
We fix a metric on $\Sigma_p$ and define $\epsilon(d)$ to be the maximal diameter of a component of $p_0^{-1}(B(s,d))$ for $s\in S^1$ and the standard metric on $S^1$.
Clearly, $\lim_{d\to 0} \epsilon(d)=0$. Since the balls $B(x,r)$
in a geodesic metric space are connected, we obtain  $$diam f(B(x,r))\le\epsilon(diam(\phi(B(x,r))).$$ Then the 
slow oscillation of $\phi$  implies that $f $ is slow oscillating. Hence there is a continuous extension $\bar f:\bar X\to\Sigma_p$ of $f$. Then $\phi=p_0\bar f$ and, hence, $\alpha=\bar f^*(1)$ where
$1\in\mathbb Z[\frac{1}{p}]=\check H^1(\Sigma_p)$. Thus, $\alpha$ is $p$-divisible.
\end{proof}
\subsection{Lipschitz extension}  The following fact is well-known. It can be proven by use of Helly's theorem and the transfinite induction~\cite{CGM},\cite{Dr}.
\begin{prop}\label{Helly}
For any metric space $X$ and any $\lambda$-Lipschitz map $f:A\to Y$, $A\subset X$, where $Y= \mathbb R$  or $Y=[a,b]$ there is a $\lambda$-Lipschitz extension $\bar f:X\to Y$.
\end{prop}

\section {Main Theorem}

\begin{thm}\label{main}
Every simply connected proper geodesic metric space $X$ admits an embedding of its Higson compactification into the product of adelic solenoids
$$F:\bar X\to\prod_{\mathcal A}\Sigma_\infty$$ that induces an isomorphism of 1-dimensional  \v{C}ech  cohomology.
\end{thm}
We denote as $[X,Y]_0$ pointed homotopy classes of maps $f:(X,x_0)\to (Y,y_0)$.
\begin{prop}\label{homotop}
Let $x_0\in X$  be a base point of a simply connected metric space $X$ such that its inclusion $x_0\to\bar X$ is a cofibration and let $0\in\Sigma_p$, $p\in\mathbb N\cup\{\infty\}$, and $1\in S^1$  be the corresponding base points. Then 
$$[\bar X,\Sigma_p]_0=[\bar X,S^1]_0=\check H^1(\bar X).$$
\end{prop}
\begin{proof}
The map $p_0:Z\to S^1$ defines the map $\Psi:[\bar X,\Sigma_p]_0\to[\bar X,S^1]_0$ as $\Psi([f])=[fp_0]$ for the pointed homotopy class $[f]$ of a map $f:(\bar X,x_0)\to(\Sigma_p,0)$. We define the inverse map $\Phi$ as follows.
Given a map $\phi:(\bar X,x_0)\to (S^1,1)$ we consider
its lift $\bar f:(\bar X,x_0)\to( \Sigma_p,0)$ with respect to $p_0$ constructed in the proof of Proposition~\ref{kees}.  If $\phi$ has a pointed homotopy to a map $\phi'$, then the lifting of that homotopy will be a pointed homotopy, since the fiber $p_0^{-1}(1)$ is disconnected. This proves  $\Phi$ is well-defined. Clearly, $\Psi\Phi=1$.
The uniqueness of lift of connected space $X$ with the prescribed value $0$ for the base point implies that $f$ is the only lift of $fp_0$. Thus, $\Phi\Psi=1$.

Since $x_0\to X$ is a cofibration, we obtain $[\bar X,S^1]_0=[\bar X,S^1]=\check H^1(\bar X)$.
\end{proof}
We can choose a metrics on the solenoids $\Sigma_p$, $p\in\mathbb N\cup\{\infty\}$, in such a way that  $p_0:\Sigma_p\to S^1$ is 1-Lipschitz and  the inclusion map $\mathbb R\to\Sigma_p$  is Lipschitz.

For a slowly oscillating map $f:X\to\Sigma_p$ we denote by $\bar f:\bar X\to\Sigma_p$ its extension to the Higson corona. Thus, $\overline{g|_X}=g$ for any 
continuous map $g:\bar X\to\Sigma_p$. We note that for connected space $X$ with $x_0\in X$ any continuous map $f:(X,x_0)\to(\Sigma_p,0)$ lands in the subgroup $\mathbb R\subset\Sigma_p$.
\begin{prop}\label{reduction}
Let $X$ be  a simply connected proper metric space with a base point. 
For any elements $[f],[g]\in[\bar X,\Sigma_p]_0$,  where $\bar X$ is the Higson compactification, and for any  $n\in\mathbb N$ we have 

$n\Psi([f])=\Psi([\overline{nf|_X}])$ and 

$\Psi([f])+\Psi([g])=\Psi([\overline{f|_X+g|_X}])$.
\end{prop}
\begin{proof}
We recall that the restriction of $p_0:\Sigma_p\to S^1$ to $\mathbb R$ equals the exponential map $\exp(t)=e^{2\pi it}$. Then
$$n\Psi([f])=n[p_0f]=[(p_0f)^n]=[\overline{(p_0f|_X)^n}]=[\overline{(exp (f|_X))^n}]=$$
$$[\overline{exp(nf|_X)}]=[\overline{p_0(nf|_X)}]=[p_0\overline{nf|_X}]=\Psi([\overline{nf|_X}])$$
and $$\Psi([f])+\Psi([g])=[p_0f]+[p_0g]=[\overline{p_0f|_X}]+[\overline{p_0f|_X}]=[\overline{p_0f|_X}\cdot\overline{p_0g|_X}]=$$
$$[\overline{exp(f|_X)\cdot exp(g|_X)}]=
[\overline{exp(f|_X+g|_X)}]=
[\overline{p_0(f|_X+g|_X)}]=
\Psi([\overline{f|_X+g|_X}]).$$ Here the multiplication $\cdot$ is taken in the unit circle $S^1$.
\end{proof}

For subsets $A$ and $B$ of a metric space $(X,d)$ we denote  by $$dist(A,B)=\inf\{d(a,b)\mid a\in A,\ b\in B\}$$
the distance from $A$ to $B$.
\begin{prop}\label{separ}
Let $X$ be a proper metric space with a base point $x_0\in X$, let $c\in\nu X$ be a fixed point, and let $A,B\subset\nu X$ be two closed disjoint subsets. For any $z\in A_p\setminus\{0\}$,
$p\in\mathbb N\cup\{\infty\}$, there is a slow oscillating function
$f_z:X\to\mathbb R\subset\Sigma_p$ with the continuous extension $\bar f_z:\bar X\to\Sigma_p$ such that $f_z(x_0)=0$, $\bar f_z(c)=z$, and $\bar f_z(A)\cap \bar f_z(B)=\emptyset$. 
\end{prop}
\begin{proof}
Suppose that $c\notin B$. Then we enlarge $A$ by adding $c$ to it. We will define slowly oscillating function $f_z:X\to\mathbb R\subset\Sigma_p$ such that $\bar f_z$ takes $B$ to $0$
and takes $A$ to $z$. 

Let $A',B'\subset X$ be closed disjoint subsets with $A=Cl_{\bar X}(A')\cap\nu X$ and
$B=Cl_{\bar X}(B')\cap\nu X$. Let $$r_n=dist(A'\setminus B(x_0,2^n), B'\setminus B(x_0,2^n)).$$ In view of Proposition~\ref{Higson-2} we may assume that $r_n$ is increasig with $r_n\to\infty$.

We consider the partition of $X\setminus B(x_0,1)$ into closed annuli $A_n$ centered at $x_0$ with outer radius $2^{n+1}$ and with the inner radius $2^n$.
Consider $$C=A'\cap\bigcup_{n=1}^{\infty}A_{2n-1}\ \ \text{and}\ \ C'=A'\cap\bigcup_{n=1}^{\infty}A_{2n}.$$ Since $A'=C\cup C'$ and $c\in Cl_{\bar X}(A')$, 
either $c\in Cl_{\bar X}(C)$ or $c\in Cl_{\bar X}(C')$.
Without loss of generality we may assume the first.

 Let $z_n\in\mathbb R$ be a sequence converging to $z$ in $\Sigma_p$. We may assume that $z_n\to\infty$ in $\mathbb R$ and
 $z_n\le\min\{ \sqrt{r_{n}},2^{n-1}\}$.
We define $f_z(C_n)=z_n$ and $f_z(B')=0$. Let $$\lambda_n=\max\{\frac{1}{\sqrt{r_n}},\frac{1}{2^n}\}.$$ Note that the restriction of $f_z$ to $(C\cup B')\cap A_{2n-1}$ is $\lambda_n$-Lipschitz. 
Indeed, for $x\in B'$ and $y\in C_n$ we have $d(x,y)\ge r_n$ and 
$$|f_z(y)-f_z(x)|=z_n\le\sqrt{r_n}\le\frac{1}{\sqrt{r_n}}d(x,y).$$
By Proposition~\ref{Helly} the above restriction of $f_z$ can be extended to a $\lambda_n$-Lipschitz map $$f_z^{2n-1}:A_{2n-1}\to [0,z_n].$$ We show that the union map $$f^{2n-1}_z\cup f^{2n+1}_z:A_{2n-1}\cup A_{2n+1}\to[0,z_{n+1}]$$ is 
$\lambda_n$-Lipschitz. For if, $x\in A_{2n+1}$ and $y\in A_{2n-1}$ we have $d(x,y)\ge 2^{2n}$ and, hence,
$$|f^{2n+1}_z(x)-f^{2n-1}_z(y)|\le z_{n+1}\le 2^n=\frac{1}{2^n}2^{2n}
\le \frac{1}{2^n}d(x,y)\le\lambda_nd(x,y).$$

 By Proposition~\ref{Helly} the map $f^{2n-1}_z\cup f^{2n+1}_z:A_{2n-1}\cup A_{2n+1}\to \mathbb R$ can be extended to $\lambda_n$-Lipschitz map  $f^{2n}_z:A_{2n}\to\mathbb R$. We extend  $f_z^0$ continuously to $B(x_0,1)$ with $f_z(x_0)=0$.
The union $\cup_k f_z^k$ defines a slow oscillating function $f_z:X\to\mathbb R$ with the required properties. Indeed, since the inclusion of the subgroup  $\mathbb R\to\Sigma_p$ is a Lipschitz function, the map $f_z:X\to\Sigma_p$ is slowly oscillating. Then by Proposition~\ref{Higson} there is a continuous extension $\bar f_z:\bar X\to\Sigma_p$. Note that $\bar f_z(B)=0$. Since the sequence $z_k\subset\mathbb Z=\mathbb R\cap A_p\subset\Sigma_p$ converges to $z$, we obtain $\bar f_z(c)=\bar f_z(A)=z\ne 0$.
\end{proof}

A family of elements $\{x_\alpha\}\subset A$ of an abelian group is called linearly independent if for any finite subfamily $\{x_{\alpha_i}\}$
satisfying the equality $\sum_in_ix_{\alpha_i}=0$ with $n_i\in\mathbb Z$
it follows that $n_i=0$ for all $i$.
\begin{prop}\label{independent}
The group $A_p/\mathbb Z$, $p\in\mathbb N\cup\{\infty\}$, contains a family of linearly independent elements $\{u_\alpha\}$ of cardinality continuum.
\end{prop}
\begin{proof}
First we show it for finite $p$. The group  of $p$-adic rational numbers $\mathbb Q_p$ contains $A_p$ and $\mathbb Q$ as  subgroups with $A_p\cap\mathbb Q=\mathbb Z$. 
The group $\mathbb Q_p$ is divisible and, hence, is a vector space over $\mathbb Q$. Thus, $\mathbb Q_p=(\mathbb Q_p/\mathbb Q)\oplus\mathbb Q$. 
Let $\{v_\alpha\}\subset\mathbb Q_p$ be a basis for the summand $\mathbb Q_p/\mathbb Q$. Then for each $v_\alpha$ there is $k=k(\alpha)$ such that $p^kv_\alpha\in A_p$.
We define $u_\alpha=p^{k(\alpha)}v_\alpha$. Suppose that $\sum n_iu_i\in\mathbb Z$. Then $\sum n_ip^{k(i)}v_i\in\mathbb Z\subset\mathbb Q$. Therefore,
$n_ip^{k(i)}=0$ for all $i$. Hence $n_i=0$ for all $i$. Thus the family $\{u_\alpha\}$ is linearly independent over $\mathbb Z$.

We note that $A_\infty\cong\prod_pA_p$ where the product is taken over all prime numbers $p$. We take a family $\{z_\alpha\}\subset A_\infty$ in such  way that $\pi_p(z_{\alpha})=u_\alpha$
where $\pi_p:\prod_pA_p\to A_p$ is the projection onto the factor.
Then $\{z_\alpha\}$ forms a linearly independent family in $A_\infty/\mathbb Z$.
\end{proof}

\

{\bf Proof of Theorem~\ref{main}.} We fix points $x_0\in X$ and $c\in\nu X$. Since $X$ is a reasonably nice space, we may assume that the inclusion $x_0\to X$ is a cofibration.
Otherwise we attach to $X$ an interval $[x_0,x_1]$ along some $x_1\in X$.

Let 
$\{z_\alpha\in A_\infty\mid\alpha\in {\mathcal A}'\}$ be the family form Proposition~\ref{independent} of the cardinality of continuum $\mathfrak{c}$. 

It is not difficult to show that the weight of the Higson corona $\nu X$ is  $\mathfrak{c}$. Let $\mathcal U$ be a basis of topology of $\nu X$ of cardinality
$\mathfrak{c}$. Define $$\mathcal P=\{(U,U')\in\mathcal U\times\mathcal U\mid \bar U\cap\bar U'=\emptyset\}$$ and 
let $\gamma:\mathcal P\to \mathcal A'$ be a bijection. For every pair $(U,U')\in\mathcal P$ we apply Proposition~\ref{separ} with  $A=\bar U$, $B=\bar U'$ and $z=z_{\gamma(U,U')}$
to obtain a function $f_z:X\to\mathbb R$ that extends to a map $\bar f_z:(\bar X,x_0)\to(\Sigma_p,0)$  with $f_z(x_0)=0$, $\bar f_z(c)=z$, and $\bar f_z(A)\cap \bar f_z(B)=\emptyset$.  

We denote by $f_\alpha= f_{z_\alpha}$, $\alpha\in\mathcal A'$.
We claim that the cohomology classes $\{[p_0\bar f_{\alpha}]\}_{\alpha\in\mathcal A'}$ in $\check H^1(\bar X)$  are linearly independent. Suppose that $\sum n_i[p_0\bar f_{\alpha_i}]=0$. Then by Proposition~\ref{reduction}, 
$$\sum n_i[p_0\bar f_{\alpha_i}]=\sum n_i\Psi[\bar f_{\alpha_i}]=\sum \Psi[\overline{n_if_{\alpha_i}}]=\Psi([\overline{\sum n_if_{\alpha_i}}])=0.$$
Hence there is a rel $x_0$ homotopy of $p_0\overline{\sum n_if_{\alpha_i}}$ to the constant map to $1\in S^1$.  The lift of this homotopy is a rel $x_0$ homotopy in $\Sigma_\infty$
of $\overline{\sum n_if_{\alpha_i}}$ to the constant map to $0$. Therefore,
$\sum n_i\bar f_{\alpha_i}(c)=\sum n_iz_{\alpha_i}$ lies in the path component of $0$. Thus, $\sum n_iz_{\alpha_i}\in\mathbb Z$.
The linear independence of $\{z_\alpha+\mathbb Z\}\in A_\infty/\mathbb Z$ implies that all $n_i=0$.

We complete the cohomology classes $\{[p_0\bar f_{\alpha}]\}_{\alpha\in\mathcal A'}$ in $\check H^1(\bar X)=\oplus\mathbb Q$  to a basis $\{[\phi_\alpha]\}_{\alpha\in\mathcal A''}$
where $\phi_\alpha:(\bar X,x_0)\to (S^1, 1)$. Let $\bar f_\alpha:(\bar X,x_0)\to(\Sigma_\infty,0)$ be the lift of $\phi_\alpha$ defined by the map $\Phi$ of Proposition~\ref{homotop}.
Let $\mathcal A=\mathcal A'\cup\mathcal A''$.
Since the map $$F=(\bar f_{\alpha})_{\alpha\in\mathcal A}:\bar X\to\prod_{\mathcal A}\Sigma_\infty$$ restricted to the Higson corona $\nu X$ separates 
the disjoint closures of basis sets, it separates distinct points and, hence,  it is an embedding on $\nu X$.
We note that the restriction $F'=F|_X:X\to \prod_{\mathcal A}\mathbb R\subset \prod_{\mathcal A}\Sigma_\infty$ to $X$ lands in $\prod_{\mathcal A}\mathbb R$. 

We approximate $F$ by embedding as follows. Fix $\omega$ coordinates  $\alpha_1,\dots,\alpha_n,\dots\in\mathcal A$ and consider $\|\ \|_\infty$ metric on $\mathbb R^\omega$. We denote by $f_i=f_{\alpha_i}$.
Let $$d_i(t)=\sup_{x\in X\setminus B(x_0,t)}diam(f_i(B(x,1))).$$ Since $f_i$ is extendible to the Higson corona, it is slowly oscillating. Hence $\lim_{t\to\infty}d_i(t)=0$.
We recall that the limitation topology on a functional space $C(X,Y)$ with metrizable $Y$ is defined by the collection of open balls $B_\rho(f,1)$ with respect to all compatible metrics $\rho$ on $Y$ where the ball in $C(X,Y)$ is defined by means of the sup metric~\cite{Bo}.
We note that the functional space $C(X,\mathbb R^\omega)$ taken with the limitation topology is a Baire space~\cite{T},\cite{Bo}. This allows to prove that the
map $f=(f_i)_{i=1}^N:X\to\mathbb R^\omega$ can be approximated by embedding $g:X\to\mathbb R^\omega$ in such a way that 
$$\rho(t)=\sup_{x\in X\setminus B(x_0,t)}\|f(x)-g(x)\|_\infty\to 0$$ as $t\to\infty$.
Then the $i$th coordinate of $g$, the map $g_i:X\to\Sigma_p$, is slowly oscillating for each $i$.  Indeed, by the triangle inequality $$\lim_{t\to\infty}\sup_{x\in X\setminus B(x_0,t)}diam(g_i(B(x,1)))\le\lim_{t\to\infty} (d_i(t)+2\rho(t))=0.$$ 
Note that the extension $\bar g_i:\bar X\to\Sigma_\infty$ of $g_i$ coincides with $\bar f_i$ on $\nu X$. 

We replace the map $F$ with a new map $$F'=(\bar g_i)_{i=1}^\infty\times (\bar f_\alpha)_{\alpha\in \mathcal A\setminus\{\alpha_1,\dots,\alpha_n\dots\}}:\bar X\to\prod_{\mathcal A}\Sigma_\infty$$ which is an embedding.
\qed

\begin{remark}\label{rem}
All coordinate maps $\bar f_\alpha$ in the embedding $F:\bar X\to\prod_{\mathcal A}\Sigma_\infty$ have the property that the restriction $\bar f_\alpha$ to $X$ lands in 
the subgroup $\mathbb R\subset\Sigma_\infty$. Moreover, by the construction $\bar f_\alpha(X)\subset\mathbb R_+$.
\end{remark}

\section{Embedding into the product of Knaster continua}
The proof of Theorem~\ref{main} in the case of finite $p$ brings the following.
\begin{thm}\label{main-p}
For any $p$ every simply connected proper geodesic metric space $X$ admits an embedding of its Higson compactification into the product of $p$-adic solenoids
$$F:\bar X\to\prod_{\mathcal A}\Sigma_p$$ that induces a rational isomorphism of 1-cohomology,
$$
\check H^1(\prod_{\mathcal A}\Sigma_p)\otimes\mathbb Q\to \check H^1(\bar X)\otimes\Q=\check H^1(\bar X)=\bigoplus_{\mathcal A}\mathbb Q.
$$
\end{thm}
We call a map $f:X\to Z$ {\em essential} if every map  $g:X\to Z$   homotopic  to $f$ is surjective. We call a subset $X\subset\prod Z_\alpha$ {\em essential} if its projection on each factor $p_\alpha:X\to Z_\alpha$ is essential. Since the embedding $F:\bar X\to\prod\Sigma_p$ in Theorem~\ref{main-p} is essential, we obtain
\begin{cor}
For any $p$ every simply connected proper geodesic metric space $X$ admits an essential embedding of its Higson compactification into a product of $p$-adic solenoids.
\end{cor}

\begin{prop}\label{onto}
For any $p\in\{\infty\}\cup \mathbb N\setminus\{1\}$ any surjective map $f:Y\to\Sigma_p$ of a connected compact  Hausdorff space is essential.
\end{prop}
\begin{proof}
Assume that $f$ is homotopic to $g$ which is not onto. Then for some projection $\pi:\Sigma_p\to S^1$ onto the circle in the inverse limit presentation of $\Sigma_p$ the composition $\pi\circ g$ is not onto. Hence it is homotopic
to a constant map. The lifting of this homotopy defines a homotopy of $g$ to the disconnected fiber $A_p$ and hence to a point. Thus, we have a homtopy of $f$ to a constant map.
Since $\Sigma_p$ has more than one path component, this is impossible for a surjecive map.
\end{proof}

We define a continuum $K_p=\Sigma_p/\sim$ to be the quotient space under the identification $x\sim -x$. A theorem of Bellamy~\cite{Be} says that $K_2$ is homeomorphic to the Knaster continuum, also known as the Bucket handle continuum. We note that the subgroup $\mathbb R\subset\Sigma_p$ is taken under this identification to a path component of $K_p$ which is an injective image of the ray $\mathbb R_+$ with the initial point $q(0)$ where $q:\Sigma_p\to K_p$ is the quotient map.

Since for complex numbers the conjugation commutes with taking the power, $\bar z^n=\overline{z^n}$, we have a commutative diagram
$$
\begin{CD}
S^1 @<z^n<< S^1\\
@VqVV @VqVV\\
I @<\bar n<< I\\
\end{CD}
$$ where $q$ is the projection onto the interval, the orbit space of the conjugation, and $\bar n$ is the induced map. Note that $\bar n:I\to I$ is the $n$ times folding the interval.
Thus, there is the inverse limit presentation of $K_p$ that fits into the diagram:
$$
\begin{CD}
S^1 @<z^{n_1}<< S^1 @<z^{n_2}<< S^1 @<z^{n_3}<<\dots @<<<\Sigma_p\\
@VqVV @VqVV @VqVV @. @VqVV\\
I @<\bar n_1<< I @<\bar n_2<< I @<\bar n_3<<\dots @<<< K_p.\\
\end{CD}
$$
We identify the first interval in the bottom inverse sequence with $[0,1]$, the second with $[0,n_1]$, the third with $[0,n_1n_2]$ and so on to view the bonding maps as 'isometric' folding.
Moreover, we view each folding map $\bar n_k:[0,n_1n_2\cdots n_k]$ as  a retraction onto $[0,n_1n_2\cdots n_{k-1}]$. Thus $$\mathbb R_+=\bigcup_k[0,n_1n_2\cdots n_k]$$
is naturally embedded into $K_p$ as the path component of $q(0)$.

\begin{prop}\label{onto2}
For any $p\in\{\infty\}\cup \mathbb N\setminus\{1\}$ any surjective map $f:Y\to K_p$ of a compact Hausdorff space is essential.
\end{prop}
\begin{proof}
Let $q:\Sigma_p\to K_p$ be the quotient map. Note that $q$ is open and all preimages $q^{-1}(y)$ are 2-point sets except $q^{-1}(q(0))=0$.
We show that the compact space $Z$ in the pull-back in the diagram
$$
\begin{CD}
Z @>f'>> \Sigma_p\\
@V q'VV @VqVV\\
Y @>f>> K_p\\
\end{CD}
$$
is connected. Suppose that $Z=U_1\cup U_2$ is a disjoint union of nonempty open sets. For $i=1,2$ we define subsets  $A_i=\{y\in Y\mid (q')^{-1}(y)\subset U_i\}$. 
Since $U_i$ are closed and open and $q'$ isan open map the sets $A_1=Y\setminus q'(U_2)$ and $A_2=Y\setminus q'(U_1)$ are closed and open.  Since $q'$ has single point fibers at least one of $A_i$ is not empty, let it be $A_1$. Since $Y$ is connected, we obtain $Y=A_1$. Then $U_2=\emptyset$ which contradicts to the assumption.

Assume that $f$ is homotopic by means of  homotopy $H:Y\times I\to K_p$ to a non-surjective map $g$. Note that for $y\in Y$ with $f(y)\in\mathbb R_+$ we have 
$H(y\times I)\subset\mathbb R_+$.  We may assume that the inner metric on $\mathbb R_+$  induced by a metric on $K_p$ is the standard metric on $\mathbb R_+$.
Thus $\mathbb R_+$ makes infinitely many waves  in $K_p$ of amplitude 1.
Then the compactness of $Y$ and the continuity of $H$ imply that there is an upper bound $b$
on the diameter of a path $H(y,-):I\to\mathbb R_+$ in $\mathbb R_+$. Let $Y_0=f^{-1}([0,2b])$ and let $a\in K_p\setminus g(Y)$. We consider a 
new homotopy $H_0:Y\times I\to K_p$  defined by the  region under the graph of a function $\phi:Y\to [0,1]$
with $\phi(Y_0)=0$ and $\phi(f^{-1}(a))=1$. Thus, the homotopy $H_0$ is stationary on $Y_0$ and does not move other points through $0\in\mathbb R_+\subset K_p$.
The homotopy $H_0$ defines a homotopy $H_0':Z\times I\to K_p$. Let $Z_0=(q')^{-1}(Y_0)$.
By the homotopy lifting property of a fibration $q|: \Sigma_p\setminus\{0\}\to K_p\setminus\{0\}$ there is a lift $\bar H_0:(Z\setminus Z_0) \times I\to \Sigma_p$ of $H'_0|_{(Y\setminus Y_0)\times I}$. Since $q$ is open, the extension of $\bar H_0$ to $\bar H_1:Z\times I\to \Sigma_p$ by the stationary homotopy of $f'|_{Z_0}$ is a homotopy of
$f':Z\to \Sigma_p$ to a map with the image in a proper subset. This contradicts to Proposition~\ref{onto}.
\end{proof}

\begin{thm}\label{imbed2}
For any $p$ and any simply connected finite dimensional proper geodesic metric space $X$ its Higson compactification can be essentially embedded into the product of  continua $K_p$.
\end{thm}
\begin{proof}
The embedding follows from the fact that the family of functions $\mathcal C=\{q\bar f_\alpha\mid\alpha\in\mathcal A\}$, where $q:\Sigma_p\to K_p$ is the quotient map, separates disjoint closures of
basis sets by the argument of Proposition~\ref{separ}. 

In view of Proposition~\ref{onto2} all maps $q\bar f_\alpha$ are essential.
\end{proof}

\end{document}